\documentclass[11pt,reqno,british,a4paper,twoside]{amsart}

\usepackage[T1]{fontenc}
\usepackage[sc]{mathpazo}
\usepackage{mathtools,url,amssymb,enumitem,textcomp}
\usepackage[main=british]{babel}
\usepackage{csquotes,simpler-wick,microtype}

\providecommand{\texorpdfstring}[2]{#1}

\mathtoolsset{mathic=true}
\setlist[enumerate,1]{label=\upshape(\roman*)}

\makeatletter
\g@addto@macro\bfseries{\boldmath}
\makeatother

\theoremstyle{plain}
\newtheorem{theorem}{Theorem}[section]
\newtheorem{proposition}[theorem]{Proposition}
\newtheorem{lemma}[theorem]{Lemma}

\theoremstyle{definition}
\newtheorem{definition}[theorem]{Definition}

\numberwithin{equation}{section}

\newcommand{\Lie}[1]{\operatorname{#1}}
\newcommand{\lie}[1]{\operatorname{\mathfrak{#1}}}
\newcommand{\G}{\Lie{G}}
\newcommand{\GL}{\Lie{GL}}
\newcommand{\SL}{\Lie{SL}}
\newcommand{\SO}{\Lie{SO}}

\newcommand{\SU}{\Lie{SU}}
\newcommand{\Spin}{\Lie{Spin}}

\newcommand{\su}{\lie{su}}
\newcommand{\spin}{\lie{spin}}
\newcommand{\lt}{\lie{t}}

\newcommand{\Hodge}{\mathop{}\!{*}}
\newcommand{\hook}{\mathbin{\lrcorner}}

\newcommand{\bC}{\mathbb{C}}

\newcommand{\bR}{\mathbb{R}}
\newcommand{\bZ}{\mathbb{Z}}

\newcommand{\cU}{\mathcal{U}}

\newcommand{\tW}{\widetilde{W}}

\DeclareMathOperator{\re}{Re}
\DeclareMathOperator{\im}{Im}
\DeclareMathOperator{\diag}{diag}

\DeclareMathOperator{\Stab}{Stab}
\DeclareMathOperator{\vol}{vol}
\DeclareMathOperator{\Hol}{Hol}
\DeclareMathOperator{\adj}{adj}

\DeclareFontFamily{U}{bigeuf}{}
\DeclareFontShape{U}{bigeuf}{m}{n}{<-6>s*[1.5]eufm5
<6-8>s*[1.5]eufm7
<8->s*[1.5]eufm10}{}
\DeclareSymbolFont{bigeufletters}{U}{bigeuf}{m}{n}
\DeclareMathSymbol{\sumcic}{\mathop}{bigeufletters}{`S}

\DeclarePairedDelimiter{\Span}{\langle}{\rangle}

\DeclarePairedDelimiter{\abs}{\lvert}{\rvert}

\DeclarePairedDelimiterX{\Set}[1]{\lbrace}{\rbrace}{%
#1 }

\DeclarePairedDelimiterX{\inp}[2]{\langle}{\rangle}{#1,#2}

\newcommand{\D}[2]{\frac{\partial #1}{\partial #2}}
\newcommand{\Dsq}[2]{\frac{\partial^2 #1}{\partial #2^2}}
\newcommand{\Dsqm}[3]{\frac{\partial^2 #1}{\partial #2\partial #3}}

\newcommand{\any}{\,\cdot\,}

\newcommand{\eqbreak}[1][2]{\\&\hspace{#1em}}
\newcommand{\eqand}[1][1]{\hspace{#1em}\text{and}\hspace{#1em}}
\newcommand{\eqcond}[2][2]{\hspace{#1em}\text{#2}}

\newcommand{\hyphen}{\nobreak-\nobreak\hskip0pt}

\begin{document}

\title{Toric geometry of \( \Spin(7) \)-manifolds}

\author[T.~B. Madsen]{Thomas Bruun Madsen}

\author[A.~F. Swann]{Andrew Swann}

\address[T.~B. Madsen]{Centre for Quantum Geometry of Moduli Spaces\\
Aarhus University\\
Ny Munkegade 118, Bldg 1530\\
8000 Aarhus\\
Denmark}

\email{thomas.madsen@math.au.dk}

\address[A.~F. Swann]{Department of Mathematics, Centre for Quantum
Geometry of Moduli Spaces, and Aarhus University Centre for
Digitalisation, Big Data and Data Analytics\\
Aarhus University\\
Ny Munkegade 118, Bldg 1530\\
8000 Aarhus\\
Denmark}

\email{swann@math.au.dk}

\begin{abstract}
  We study \( \Spin(7) \)-manifolds with an effective multi\hyphen
  Hamiltonian action of a four-torus. On an open dense set, we provide
  a Gibbons-Hawking type ansatz that describes such geometries in
  terms of a symmetric \( 4\times4 \)-matrix of functions.
  This description leads to the first known \( \Spin(7) \)-manifolds
  with a rank \( 4 \) symmetry group and full holonomy.
  We also show that the multi-moment map exhibits the full orbit space
  topologically as a smooth four-manifold, containing a trivalent
  graph in \( \bR^4 \) as the image of the set of the special orbits.
\end{abstract}


\maketitle

\begin{center}
  \begin{minipage}{0.8\linewidth}
    \microtypesetup{protrusion=false} \small \tableofcontents
  \end{minipage}
\end{center}

\section{Introduction}
\label{sec:intro}

It was Berger~\cite{Beger:hol-clsf} who first realised that the Lie
group \( \Spin(7) \) could potentially arise as the holonomy group of
a non-symmetric irreducible Riemannian manifold.
A decade later, Bonan~\cite{Bonan:exc-hol} showed that such manifolds
would necessarily be Ricci-flat and come with a parallel \( 4 \)-form
of a certain algebraic type.
Subsequently the understanding that parallelness amounts to closedness
(cf.~\cite{Fernandez:Spin7-manf}) has been a powerful tool when
looking for examples.
The first \( 8 \)-manifolds with holonomy equal to \( \Spin(7) \) were
constructed in the late
1980s~\cite{Bryant:excep-hol,Bryant-S:excep-hol} by Bryant and Salamon
and since then many examples have followed, including compact ones by
Joyce~\cite{Joyce:cpt-Spin7a,Joyce:cpt-Spin7b}.
In fact, a torsion-free \( \Spin(7) \)-structure can be obtained from
any closed spin \( 7 \)-manifold~\cite{Crowley-Nordstrom:G2-inv}, but
in general this will neither be complete nor have full holonomy.

Despite considerable advances in the field, it still remains a
challenge to construct new complete examples with holonomy equal to
\( \Spin(7) \).
A natural way to approach this problem systematically is to consider
examples with a specific type of symmetry.
Indeed, a key point in constructing the first examples was to apply
cohomogeneity one symmetry techniques, and in~\cite{Madsen:Spin7andT3}
the first author gave a description of \( \Spin(7) \)-manifolds with
\( T^3 \)-symmetry (these are characterised in terms of certain
tri-symplectic \( 4 \)-manifolds).

From a toric viewpoint, it is natural to consider
\( \Spin(7) \)-manifolds with a multi-Hamiltonian action of~\( T^4 \),
this the critical rank making the dimensions of the leaf space
\( M/T^{k} \) and the target space of the multi-moment map match.
As we will see, this gives the type of behaviour we expect from toric
Ricci-flat geometries (cf.~\cite{Dancer-S:hK,Madsen-S:toric-G2}).
We introduce the notion of a \emph{toric \( \Spin(7) \)-manifold} to
be a (torsion-free) \( \Spin(7) \)-manifold \( (M,\Phi) \) that comes
with an effective multi-Hamiltonian action of a rank four torus.
As we explain in~\S\ref{sec:eff-action}, this implies that we have a
multi-moment map \( \nu\colon M\to\bR^4 \) that exhibits an open dense
subset \( M_0\subset M \) as a principle \( T^4 \) bundle over an open
subset \( \cU\subset\bR^4 \).
On this regular part, we derive an analogue of the Gibbons-Hawking
ansatz.
What is needed in this case is a smooth positive definite section
\( V\in\Gamma(\cU,S^2(\bR^4)) \) satisfying a pair of PDEs.
One of these is a divergence-free condition and the other is a
quasi-linear elliptic second order PDE\@.
These equations are natural when one considers differential operators
that are invariant, up to scaling, under the \( \GL(4,\bR) \) action
resulting from changing the basis of the Lie algebra \( \lt^4 \) of
the torus \( T^4 \).

In order to achieve a complete understanding of \( M \), we address
the behaviour near singular orbits in~\S\ref{sec:sing-beh}.
It turns out that the only special orbits are circles and two-tori.
Describing the flat model associated with each singular orbit enables
us to show that the orbit space \( M/T^4 \) is homeomorphic to a
smooth \( 4 \)-manifold, with a global local homeomorphism induced by
the multi-moment map \( \nu \).
It also follows that the image of the singular orbits in \( M/T^4 \)
consists of trivalent graphs in \( \bR^4 \).

Whilst there are currently no known complete full holonomy
\( \Spin(7) \)-manifolds with a rank four symmetry group, our approach
produces the first known incomplete examples,
see~\S\ref{sec:diag-case}.

\subsection*{Acknowledgements}
AFS was partially supported by the Danish Council for Independent
Research~\textbar~Natural Sciences project DFF - 6108-00358.
Both authors partially supported by the Danish National Research
Foundation grant DNRF95 (Centre for Quantum Geometry of Moduli Spaces
- QGM).

\section{Multi-Hamiltonian
\texorpdfstring{\( \Spin(7) \)}{Spin7}-manifolds}
\label{sec:eff-action}

Let \( M \) be a connected \( 8 \)-manifold.
A \( \Spin(7) \)-structure on \( M \) is determined by a \( 4 \)-form
\( \Phi \) that is pointwise linearly equivalent to the form
\( \Phi_0=e^0\wedge \varphi_0+\Hodge_{\varphi_0}\varphi_0 \), where
\begin{equation*}
  \varphi_0 =
  e^{123}-e^1(e^{45}+e^{67})-e^2(e^{46}+e^{75})-e^3(e^{47}+e^{56});
\end{equation*}
\( E_0,\dots, E_7 \) is a basis of \( V\cong \bR^8 \), and
\( e^0,\dots, e^7 \) is its dual basis of \( V^* \).
Occasionally, we shall refer to the basis \( E_0,\dots, E_7 \) (and
its dual) as being \emph{adapted}.

The \( \GL(V) \)-stabiliser of \( \Phi_0 \) is the compact
\( 21 \)-dimensional Lie group \( \Spin(7)\subset\SO(V) \).
In fact, \( \Phi_0 \) uniquely determines both the inner product
\( g_0 = \sum_{j=0}^7e_j^2 \) and volume element
\( \vol_0 = e^{01234567} \) (see~\cite{Bryant:excep-hol,Spiro:deform-exphol}).
Correspondingly, \( \Phi \)~determines a metric~\( g \) and a volume
form on~\( M \).

Following standard terminology, we say that \( (M,\Phi) \) is a
\emph{\( \Spin(7) \)-manifold} if the \( \Spin(7) \)-structure is
torsion-free, hence the (restricted) holonomy group \( \Hol_0(g) \) is
contained in \( \Spin(7)\subset\SO(8) \).
This implies \( g \) is Ricci-flat.
It is well-known~\cite{Fernandez:Spin7-manf} that being torsion-free,
in this context, is equivalent to the condition that \( \Phi \) is
closed.

We are interested in \( \Spin(7) \)-manifolds that come with an
effective action of a four-torus~\( T^4 \) on~\( M \) that
preserves~\( \Phi \), hence also the metric \( g \).
This furnishes a a Lie algebra anti-homomorphism
\begin{equation}
  \label{eq:inf-act}
  \xi\colon \bR^4\cong\lt^4\to\mathfrak{X}(M).
\end{equation}
In the following, we shall occasionally use \( \xi_p \) to denote the
image of \( \xi \) at \( p \in M \), which is a subspace of
\( T_p M \) of dimension at most~\( 4 \).

\begin{definition}[{\cite[Def.~3.5]{Madsen-S:mmmap2}}]
  Let \( N \) be a manifold equipped with a closed \( 4
  \)-form~\( \alpha \), and \( G \)~an Abelian Lie group acting
  on~\( N \) preserving~\( \alpha \).
  A \emph{multi-moment map} for this action is an invariant map
  \( \nu\colon N \to \Lambda^3\lie g^* \) such that
  \begin{equation*}
    d\inp{\nu}{W}=\xi(W)\hook\alpha,
  \end{equation*}
  for all \( W\in\Lambda^3\lie g \); here
  \( \xi(W)\in\Gamma(\Lambda^3 TM) \) is the unique multi-vector
  determined by~\( W \) via~\( \xi \).
\end{definition}

The \( T^4 \)-action being multi-Hamiltonian for \( \Phi \) implies
that \( \Phi|_{\Lambda^4\xi}\equiv0 \) (cf.\ \cite[Lemma
2.5]{Madsen-S:toric-G2}).
For \( p\in M_0 \), consider an orthonormal
\( X_0,X_1,X_2,X_3 \in \xi_p \) with \( \hat\theta_i \),
\( i=0,1,2,3 \) dual to \( X_0,X_1,X_2,X_3 \):
\( \hat\theta_i(X_j)=\delta_{ij} \) and \( \hat\theta_i(X)=0 \) for
\( X\in\Span{X_0,X_1,X_2,X_3}^\perp \).
Next, let us we denote by \( \alpha_i \) the \( 1 \)-forms
\begin{equation*}
  \alpha_i=(-1)^{i}X_j\wedge X_k\wedge X_\ell\hook\Phi,
\end{equation*}
where \( (ijk\ell)=(0123) \), as cyclic permutations.

As \( \Spin(7) \) acts transitively on the sphere \( S^7 \), we may
take \( X_0=E_0 \) at \( p \).
Now \( \varphi=X_0\hook\Phi \) is a \( \G_2 \)-form, isotropic for
\( X_1,X_2,X_3 \).
Our analysis of \( \G_2 \)-forms~\cite{Madsen-S:toric-G2} shows that
we may take these \( X_{i} \) to be \( E_{5},E_{6},E_{7} \) and so we
get:
\begin{equation}
  \label{eq:Spin7formsa}
  \Phi=\hat\theta_0\wedge \varphi+\Hodge_{\varphi}\varphi,
\end{equation}
where
\begin{gather*}
  \varphi = \alpha_{123} +
  \alpha_1(\alpha_0\hat{\theta}_1-\hat{\theta}_{23}) +
  \alpha_2(\alpha_0\hat\theta_2-\hat\theta_{31})
  + \alpha_3(\alpha_0\hat\theta_3-\hat\theta_{12}),\\
  \Hodge\varphi = \hat\theta_{123}\alpha_0 +
  \alpha_{23}(\alpha_0\hat{\theta}_1-\hat{\theta}_{23}) +
  \alpha_{31}(\alpha_0\hat\theta_2-\hat\theta_{31}) +
  \alpha_{12}(\alpha_0\hat\theta_3-\hat\theta_{12}).
\end{gather*}

Examining the possible isotropy groups, we have the following
surprisingly clean result.

\begin{lemma}
  \label{lem:conn-isotr}
  Suppose \( T^4 \) acts effectively on a manifold \( M \) with
  \( \Spin(7) \)-structure \( \Phi \) so that the orbits are
  isotropic, \( \Phi|_{\Lambda^4\xi_p} = 0 \).
  Then each isotropy group is connected and of dimension at most two;
  hence trivial, a circle or~\( T^2 \).
\end{lemma}

\begin{proof}
  As \( \Spin(7) \) has rank \( 3 \), an isotropy group for
  \( T^{k} \) is of dimension at most~\( 3 \).
  It follows that the \( T^4 \)-orbits are at least one-dimensional.
  In particular, there is always one isotropy invariant direction.
  Hence, the isotropy group is a subgroup of \( \G_2 \).
  But \( \G_{2} \) has rank~\( 2 \), so the isotropy group is at most
  \( 2 \)-dimensional.
  Now as in~\cite{Madsen-S:toric-G2}, the isotropy group is seen to be
  a maximal torus in \( \SU(r) \), \( r=1,2,3 \), so is connected and
  either trivial, a circle or \( T^2 \), as asserted.
\end{proof}

In particular, we have that \( M_{0} \) is the total space of a
principal \( T^{4} \)-bundle.

\section{Toric \texorpdfstring{\( \Spin(7) \)}{Spin7}-manifolds: local
characterisation}
\label{sec:Spin7can-form}

Following the discussion in~\S\ref{sec:eff-action}, we introduce the
following terminology:

\begin{definition}
  \label{def:toric}
  A \emph{toric \( \Spin(7) \)-manifold} is a torsion-free
  \( \Spin(7) \)-manifold \( (M,\Phi) \) with an effective
  multi-Hamiltonian action of \( T^4 \).
\end{definition}

The main aim of this section is to derive a Gibbons-Hawking type
ansatz~\cite{Gibbons-Hawking1,Gibbons-Hawking2} for toric
\( \Spin(7) \)-manifolds: we obtain a local form for a toric
\( \Spin(7) \)-structure on \( M_{0} \) and characterise the
torsion-free condition in these terms.

So assume \( (M,\Phi) \) is a toric \( \Spin(7) \)-manifold.
Let \( U_0,U_1,U_2,U_3 \) be infinitesimal generators for the
\( T^4 \)-action; these give a basis for \( \xi_p\leqslant T_p M \)
for each~\( p\in M_0 \).
Denote by \( \theta = (\theta_0,\theta_1,\theta_2,\theta_3)^t \) the
dual basis of \( \xi_p^*\leqslant T_p^*M \):
\begin{equation*}
  \theta_i(U_j)=\delta_{ij} \eqand
  \theta(X)=0 \eqcond[1]{for all \( X\perp U_0,U_1,U_2,U_3 \).}
\end{equation*}
As shorthand notation, we shall write \( \theta_{ab} \) for
\( \theta_a \wedge \theta_b \), and so forth.

Let \( \nu = (\nu_0,\nu_1,\nu_2,\nu_3)^t \) be the associated
multi-moment map; its components satisfy
\begin{equation*}
  d\nu_i
  = (-1)^{i}U_j\wedge U_k\wedge U_{\ell}\hook\Phi
  = (-1)^{i}(U_j\times U_k\times U_\ell)^\flat,
  \qquad(ijk\ell)=(0123).
\end{equation*}
It follows that \( d\nu \) has full rank on~\( M_0 \) and induces a
local diffeomorphism \( M_0/T^4\to\bR^4 \).
We define a \( 4\times4 \)-matrix \( B \) of inner products given by
\begin{equation*}
  B_{ij}=g(U_i,U_j).
\end{equation*}
On~\( M_0 \) we set \( V=B^{-1}=\det(B)^{-1}\adj(B) \).

Using the above notation, we have the following local expression for
the \( \Spin(7) \)-structure:

\begin{proposition}
  \label{prop:Spin7-can-form}
  On \( M_0 \), the \( 4 \)-form \( \Phi \) is
  \begin{equation*}
    \begin{split}
    \Phi &= \det(V)\sumcic_{ijk\ell}(-1)^{i}\theta_i\wedge
    d\nu_{jk\ell}+
      \sumcic_{ijk\ell}(-1)^{\ell}\theta_{ijk}\wedge d\nu_{\ell}
      \eqbreak
    +\tfrac1{2\det(V)}(d\nu^t\adj(V)\theta)^2.
    \end{split}
  \end{equation*}
  The associated \( \Spin(7) \)-metric is given by
  \begin{equation}
    \label{eq:Spin7-metr}
    g = \tfrac1{\det(V)}\theta^t\adj(V)\theta+d\nu^t\adj(V)d\nu.
  \end{equation}
\end{proposition}

\begin{proof}
  We start by choosing an auxiliary symmetric matrix \( A>0 \) such
  that \( A^2=B^{-1}=V \) which is possible as \( B \) is positive
  definite.
  Then we set \( X_i=\sum_{j=0}^3A_{ij}U_j \) and observe that
  \begin{equation*}
    g(X_i,X_j)=(ABA)_{ij}=(A^2B)_{ij}=\delta_{ij},
  \end{equation*}
  showing that the quadruplet \( (X_0,X_1,X_2,X_3) \) is orthonormal.
  It follows that we can apply the formula~\eqref{eq:Spin7formsa} for
  \( \Phi \).  Explicitly,
  \begin{equation}
    \label{eq:Spin7-exp}
    \begin{split}
      \Phi
      &= \hat\theta_0\wedge\alpha_{123}
        - \hat\theta_1\wedge\alpha_{230}
        + \hat\theta_2\wedge\alpha_{301}
        - \hat\theta_3\wedge\alpha_{012}
        \eqbreak
        + \hat\theta_{123}\wedge\alpha_0
        - \hat\theta_{230}\wedge\alpha_1
        + \hat\theta_{301}\wedge\alpha_2
        - \hat\theta_{012}\wedge\alpha_3
        \eqbreak
        - \hat\theta_{01}\wedge\alpha_{01}
        - \hat\theta_{02}\wedge\alpha_{02}
        - \hat\theta_{30}\wedge\alpha_{30}
        \eqbreak
        - \hat\theta_{23}\wedge\alpha_{23}
        - \hat\theta_{31}\wedge\alpha_{31}
        - \hat\theta_{12}\wedge\alpha_{12}.
    \end{split}
  \end{equation}
  We make the identification \( \bR^4\cong \Lambda^3\bR^4 \) via
  contraction with the standard volume form.
  Then by letting \( \Lambda^3A \) denote the induced action of \( A \)
  on \( \Lambda^3\bR^4 \), we get the identity
  \begin{equation*}
    \Lambda^3A = \det(A)A^{-1}.
  \end{equation*}
  As a result, we get
  \begin{equation*}
    \alpha=(\Lambda^3A)d\nu \eqand
    \hat\theta=A^{-1}\theta=\tfrac1{\det(A)}(\Lambda^3A)\theta.
  \end{equation*}
  The asserted formula for \( \Phi \) then follows as the first line
  of~\eqref{eq:Spin7-exp} equals
  \( \det(V)\sumcic_{ijk\ell}(-1)^i\theta_i\wedge d\nu_{jk\ell}, \)
  the second line reads
  \( \sumcic_{ijk\ell}(-1)^\ell\theta_{ijk}\wedge d\nu_{\ell} \) and
  the third line may be expressed as
  \( \tfrac1{2\det(V)}(d\nu^t\adj(V)\theta)^2 \).

  Now the expression from the metric follows by direct
  computation:
  \begin{equation*}
    \begin{split}
      g
      &= \hat\theta^t\hat\theta+\alpha^t\alpha
        = (A^{-1}\theta)^t A^{-1}\theta
        + (\Lambda^2Ad\nu)^t\Lambda^2Ad\nu\\
      &=\theta^t\bigl(\tfrac1{\det(V)}\adj(V)\bigr)\theta
        + d\nu^t\adj(V)d\nu.
        \qedhere
    \end{split}
  \end{equation*}
\end{proof}

We remark that there is a natural action of \( \GL(4,\bR) \),
corresponding to changing basis of \( \lt^4 \).
This action can be useful when looking for invariants, up to scaling,
and may also be used to simplify arguments as it allows us to assume
that \( V \) is diagonal or the identity matrix at a given point,
assuming only the \( \bR^4 = \widetilde{T^4} \) action is of
relevance.

\subsection{The torsion-free condition}
\label{sec:torsion-free}

The \( \Spin(7) \)-structure featuring in
Proposition~\ref{prop:Spin7-can-form} is generally not torsion-free.
To address this, we need to compute \( d\Phi \), which involves
determining the exterior derivative of \( \theta \).
By our observations in~\S\ref{sec:eff-action}, we may think of
\( \theta \) as a connection \( 1 \)-form and its exterior derivative
\begin{equation*}
  d\theta=\omega=(\omega_0,\omega_1,\omega_2,\omega_3)^t
\end{equation*}
is therefore a curvature \( 2 \)-form (and so represents an integral
cohomology class).
In terms of our parameterisation for the base space, via the
multi-moment map, we can express the curvature components of
\( \omega \) as
\begin{equation*}
  \omega_\ell=\sum_{0\leqslant i<j\leqslant3}z^{ij}_\ell d\nu_{ij}.
\end{equation*}
We collect the curvature coefficients in four skew \( 4\times 4 \)
matrices \( Z_\ell=(z^{ij}_\ell) \).

Closedness of \( \Phi \) implies that the curvature matrices
\( Z_\ell \) are determined via \( V \) and \( dV \).
In addition, the following equations must hold
\begin{equation}
  \label{eq:div-free}
  \sum_{i=0}^3 \D{V_{ij}}{\nu_i} = 0,\qquad j=0,1,2,3.
\end{equation}
We refer to this first order underdetermined elliptic PDE system as
the \enquote{divergence-free} condition.

The explicit expressions for the curvature coefficients are
\begin{gather}
  \label{eq:Z-coeff-a}
  z^{\ell i}_\ell= \sum_{p=0}^3V_{pj}\D{V_{\ell
  k}}{\nu_p}-V_{pk}\D{V_{\ell j}}{\nu_p}, \\
  \label{eq:Z-coeff-b}
  \begin{split}
    z^{ij}_\ell
    &= V_{\ell k}\D{V_{\ell i}}{\nu_i}
      + V_{\ell k}\D{V_{\ell j}}{\nu_j}
      + V_{\ell k}\D{V_{\ell k}}{\nu_k}
      + \sum_{p=0}^3V_{\ell p}\D{V_{\ell k}}{\nu_p}
      \eqbreak
      - V_{ik}\D{V_{\ell\ell}}{\nu_i}
      - V_{jk}\D{V_{\ell\ell}}{\nu_j}
      - V_{kk}\D{V_{\ell\ell}}{\nu_k},
  \end{split}
\end{gather}
where if \( \ell=0 \) then \( (ijk)=(123) \); if \( \ell=1 \) then
\( (ijk)=(320) \); if \( \ell=2 \) then \( (ijk)=(013) \); if
\( \ell=3 \) then \( (ijk)=(021) \).

There are exactly \( 10 \) additional equations, arising from the
condition \( d\omega=0 \).
Using~\eqref{eq:div-free}, these equations can be expressed in the
form of a second order quasilinear elliptic PDE without zeroth order
terms:
\begin{equation}
  \label{eq:elliptic}
  L(V) + Q(dV) = 0.
\end{equation}
In the above, the operator \( L \) is given by:
\begin{equation*}
  L=\sum_{i,j=0}^{3}V_{ij}\Dsqm{}{\nu_i}{\nu_j}.
\end{equation*}
So \( L \) has the same principal symbol as the Laplacian for the
metric \( d\nu^t B d\nu \), which is conformally the same as the
restriction of the \( \Spin(7) \)-metric~\eqref{eq:Spin7-metr} to the
horizontal space.

The operator \( Q \) is the quadratic form in \( dV \) that is given
explicitly by
\begin{equation*}
  \begin{split}
    Q(dV)_{ii}
    &=-2 \D{V_{ij}}{\nu_i}\D{V_{ii}}{\nu_j}
      - 2\biggl(\D{V_{ij}}{\nu_j}\biggr)^2
      - 2\D{V_{ik}}{\nu_i}\D{V_{ii}}{\nu_k}
      - 2\D{V_{ij}}{\nu_k}\D{V_{ki}}{\nu_j}
      \eqbreak
      - 2\D{V_{ij}}{\nu_j}\D{V_{ki}}{\nu_k}
      - 2\biggl(\D{V_{ik}}{\nu_k}\biggr)^2
      - 2\D{V_{ii}}{\nu_\ell}\D{V_{\ell i}}{\nu_i}
      - 2\D{V_{ij}}{\nu_\ell}\D{V_{\ell i}}{\nu_j}
      \eqbreak
      - 2\D{V_{ik}}{\nu_\ell}\D{V_{\ell i}}{\nu_k}
      - 2\D{V_{ij}}{\nu_j}\D{V_{\ell i}}{\nu_\ell}
      - 2\D{V_{ik}}{\nu_k}\D{V_{\ell i}}{\nu_\ell}
      - 2\biggl(\D{V_{i\ell}}{\nu_\ell}\biggr)^2,
  \end{split}
\end{equation*}
where \( (ijk\ell)=(0123) \), and
\begin{equation*}
  \begin{split}
    \MoveEqLeft[0]
    Q(dV)_{ij}\\
    &=\D{V_{ij}}{\nu_i}\D{V_{ji}}{\nu_j}
      + \D{V_{ij}}{\nu_i}\D{V_{ik}}{\nu_k}
      + \D{V_{ij}}{\nu_i}\D{V_{i\ell}}{\nu_\ell} %
      + \D{V_{ij}}{\nu_j}\D{V_{jk}}{\nu_k}
      + \D{V_{ij}}{\nu_j}\D{V_{j\ell}}{\nu_\ell}
      - \D{V_{ij}}{\nu_k}\D{V_{ki}}{\nu_i} \eqbreak[1]
      - \D{V_{ij}}{\nu_\ell}\D{V_{\ell i}}{\nu_i}
      - \D{V_{ii}}{\nu_j}\D{V_{jj}}{\nu_i}
      - \D{V_{ik}}{\nu_j}\D{V_{jj}}{\nu_k} %
      - \D{V_{i\ell}}{\nu_j}\D{V_{jj}}{\nu_\ell}
      - \D{V_{ii}}{\nu_k}\D{V_{jk}}{\nu_i}
      - \D{V_{ij}}{\nu_k}\D{V_{jk}}{\nu_j} \eqbreak[1]
      - \D{V_{ik}}{\nu_k}\D{V_{jk}}{\nu_k}
      - \D{V_{i\ell}}{\nu_k}\D{V_{jk}}{\nu_\ell}
      - \D{V_{ii}}{\nu_\ell}\D{V_{j\ell}}{\nu_i} %
      - \D{V_{ij}}{\nu_\ell}\D{V_{j\ell}}{\nu_j}
      - \D{V_{ik}}{\nu_\ell}\D{V_{j\ell}}{\nu_k}
      - \D{V_{i\ell}}{\nu_\ell}\D{V_{j\ell}}{\nu_\ell},
  \end{split}
\end{equation*}
where \( i,j,k,\ell\in\Set{0,1,2,3} \) are distinct numbers.

In summary, we see that the torsion-free condition determines the
curvature matrices \( Z_\ell \) together with four first order
equations and ten second order equations.
Hence, we have the following way to locally characterise toric
\( \Spin(7) \)-manifolds.

\begin{theorem}
  \label{thm:toric-Spin7-charac}
  Any toric \( \Spin(7) \)-manifold can be expressed in the form of
  Proposition~\ref{prop:Spin7-can-form} on the open dense subset of
  principal orbits for the \( T^4 \)-action.

  Conversely, given a principal \( T^4 \)-bundle over an open subset
  \(\cU\subset \bR^4 \), parameterised by
  \( \nu=(\nu_0,\nu_1,\nu_2,\nu_3) \), together with
  \( V\in\Gamma(\cU,S^2(\bR^4)) \) that is positive definite at each
  point.
  Then the total space comes equipped with a \( \Spin(7) \)-structure
  of the form given in Proposition~\ref{prop:Spin7-can-form}.
  This structure is torsion-free, hence toric, if and only if the
  curvature matrices \( Z_\ell \) are determined by~\( V \) via
  \eqref{eq:Z-coeff-a} and~\eqref{eq:Z-coeff-b}, respectively, and
  \( V \)~satisfies the divergence-free condition~\eqref{eq:div-free}
  together with the quasilinear second order elliptic
  system~\eqref{eq:elliptic}.  \qed
\end{theorem}

To conclude this section, we remark that it is possible to integrate
the divergence-free equations~\eqref{eq:div-free} to obtain a
potential.  However, the correspondence is not elliptic.

\begin{proposition}
  Assume that \( V\in\Gamma(\cU,S^2(\bR^4)) \) satisfies the
  divergence-free equations~\eqref{eq:div-free}, with
  \( \cU\subset\bR^4 \) simply connected.
  Then there exists a matrix function \( A\in\Gamma(\cU,M_6(\bR)) \)
  whose second derivatives determine \( V \).
  More precisely, indexing \( \bR^{6} = \Lambda^{2}\bR^{4} \) by
  \( ij = i \wedge j \), for \( i,j \in \Set{0,1,2,3} \), there is an
  \( A_{ij,k\ell} \) satisfying
  \( A_{ij,\Hodge k\ell} = A_{k\ell,\Hodge ij} \) such that
  \begin{equation}
    \label{eq:potential}
    V_{ab}
    = \sum_{k,\ell=0}^{3} \Dsqm{A_{ak,\Hodge b\ell}}{\nu_{k}}{\nu_{\ell}}.
  \end{equation}
\end{proposition}

\begin{proof}
  We begin by noting that the divergence-free equation can be written
  more concisely as
  \begin{equation*}
    d \Hodge_4 (Vd\nu) = 0,
  \end{equation*}
  where \( \nu=(\nu_0,\nu_1,\nu_2,\nu_3)^t \) and \( \Hodge_4 \) is
  the flat Hodge star operator with respect to the
  \( \nu \)-coordinates.
  As \(\cU \) is simply connected, we deduce that \( \Hodge_4 Vd\nu \)
  is exact, i.e.,
  \begin{equation*}
    Vd\nu = \Hodge_4 d(W \kappa)
  \end{equation*}
  for some \( W\in \Gamma(\cU,M_{4\times6}(\bR)) \) and
  \( \kappa=(d\nu_{01},d\nu_{02},\dots,d\nu_{23})^t \).

  Now, using the symmetry of \( V \), we find that
  \( \tW\in \Gamma(\cU,M_{6\times4}(\bR)) \) given by \(
  \tW_{pq,i} = W_{q,\Hodge pi} - W_{p,\Hodge qi}\)  satisfies:
  \begin{equation*}
    d\Hodge_4 (\tW d\nu) = 0.
  \end{equation*}
  Note \( 2W_{i,\Hodge
  pq}  = \tW_{pq,i} - \tW_{qi,p} - \tW_{ip,q} \), so \( \tW \)
  determines~\( W \) uniquely.
  As before, the differential equation can be integrated.
  Indeed, we can find a section \( A\in\Gamma(\mathcal U,M_6(\bR)) \)
  such that \( \tW d\nu = \Hodge_4d(A \kappa) \).

  In conclusion, \( V \) can be expressed in terms of the second
  derivatives of the entries of~\( A \), with the explicit expressions
  given by~\eqref{eq:potential}.
\end{proof}

\subsection{Natural PDEs}
\label{sec:natural-eq}

We have already remarked that in our description of toric
\( \Spin(7) \)-manifolds there is an action of \( \GL(4,\bR) \)
corresponding to a different choice of generators for
\( \lt^{4} \cong \bR^{4} \).
As for toric \( \G_2 \)-manifolds
(cf.~\cite[\S3.2]{Madsen-S:toric-G2}), this action furnishes a way of
approaching equation~\eqref{eq:elliptic}, by understanding how the
operators \( L \) and \( Q \) transform under \( \GL(4,\bR) \).

It is useful use the identification
\( \GL(4,\bR)\cong(\bR^{\times}\times\SL(4,\bR))/\bZ_{2} \), where
\( \bZ_{2} \) is generated by \( -1_{4} \), and accordingly express
irreducible representations as \( \ell^k\Gamma_{a,b,c} \), where
\( \Gamma_{a,b,c} \) is an irreducible representation of
\( \SL(4,\bR) \), and \( \ell \) is the standard one-dimensional
representation of \( \bR^{\times}\to \bR\setminus\Set{0} \):
\( t \mapsto t \).
As an example, this means that we have for \( p\in M_{0} \) that
\( \xi_p=\ell^1\Gamma_{0,0,1} \).

Now let \( U = (\bR^4)^* = \ell^{-1}\Gamma_{1,0,0} \), viewed as a
representation of \( \GL(4,\bR) \).
Then \( V \in S^2(U)=\ell^{-2}\Gamma_{2,0,0} \).
The collection of first order partial derivatives
\( V^{(1)} = (V_{ij,k}) = (\partial V_{ij}/\partial\nu_k) \) is then
an element of
\( S^2(U)\otimes \ell^{-4} U^* = \ell^{-5}\Gamma_{2,0,0} \otimes
\Gamma_{0,0,1} \), since \( d\nu \) transforms as an element of
\( \Lambda^{3}U^{*} = \ell^{3}\Gamma_{1,0,0} = \ell^{4}U \).
This tensor product decomposes as
\begin{equation*}
  S^2(U)\otimes \ell^{-4} U^* = \ell^{-5} \Gamma_{1,0,0} \oplus
  \ell^{-5} \Gamma_{2,0,1},
\end{equation*}
with the projection to \( \Gamma_{1,0,0} \) being given by the
contraction
\( S^2(\Gamma_{1,0,0})\otimes \Gamma_{0,0,1}\to \Gamma_{1,0,0} \), and
\( \Gamma_{2,0,1} \) denoting the kernel of this map.
The divergence-free equation~\eqref{eq:div-free} simply says that this
contraction is zero, and so
\( V^{(1)} \in \ell^{-5} \Gamma_{2,0,1} \).

The operator \( Q \) is a symmetric quadratic operator on
\( V^{(1)} \) with values in \( S^2(U) \).
Hence, we may think of \( Q(dV) \) as an element of the space
\( \ell^{8}S^2(\Gamma_{2,0,1})^{*} \otimes S^2(\Gamma_{1,0,0}) \).
This space contains exactly one submodule that is trivial as an
\( \SL(4,\bR) \)-module, since \( S^2(\Gamma_{1,0,0})^{*} \) is a
submodule of \( S^2(\Gamma_{2,0,1})^{*} \).
Direct computations show that \( Q(dV) \) belongs to~\( \ell^{8} \).

In a similar way, we can address the second order terms
in~\eqref{eq:elliptic}.
We have
\( V^{(2)} = (V_{ij,k\ell}) \in R = (S^2(U) \otimes S^2(\ell^{-4}U^*))
\cap (\ell^{-8}\Gamma_{2,0,1}\otimes \Gamma_{0,0,1}) =
\ell^{-8}\Gamma_{1,0,1} + \ell^{-8}\Gamma_{2,0,2}\).
Now, \( L(V) \) is built from a product of \( V \) with \( V^{(2)} \)
and takes values in \( S^2(U) \).
So \( L(V) \in S^2(U)^{*}\otimes R^{*} \otimes S^2(U) \).
In this case, there are two submodules isomorphic to \( \ell^{8} \),
but only one that appears in \( L(V) \), corresponding to the
contractions
\begin{equation*}
  \wick[offset=1.3em]{\c1{S^2(U^{*})}\otimes\bigl( \c2{S^2(U^{*})}\otimes
  \c1{S^2(\ell^{4}U)}\bigr)\otimes \c2{S^2(U)}}\to \ell^{8}.
\end{equation*}
Contracting in this way seems to be the more natural choice.

Summing up the above discussion, \( L \) and \( Q \) are preserved up
to scale by the \( \GL(4,\bR) \) change of basis, and this specifies
\( Q \) uniquely.
This is completely analogous to what happens in the \( \G_2 \) setting
(see~\cite[Prop.~3.7]{Madsen-S:toric-G2}).

\begin{proposition}
  Under the action of \( \GL(4,\bR) \), \( L(V) \) and \( Q(dV) \)
  transform as elements of \( \ell^{8} \).
  Moreover, up to scaling, \( Q \) is the unique \( S^2(U) \)-valued
  quadratic form in \( dV \) with this property.  \qed
\end{proposition}

\section{Behaviour near singular orbits}
\label{sec:sing-beh}

We now want to address the singular behaviour of toric
\( \Spin(7) \)\hyphen manifolds.
As non-trivial stabilisers, we have tori of dimension \( 2 \) or
\( 1 \).

For a two-dimensional stabiliser, the flat model is
\( M = T^{2} \times \bC^{3} \), with local coordinates \( (x,y) \) on
\( T^{2} = \bR^{2}/2\pi\bZ^{2} \), \( z_{j} = x_{j}+iy_{j} \)
on~\( \bC^{3} \).
Putting \( e^{0} = dx \), and using the standard \( \varphi \) on
\( S^{1} \times \bC^{3} \) as in~\cite{Madsen-S:toric-G2}, we have
\begin{equation*}
  \begin{split}
    \Phi
    &= \tfrac{i}{2}dx\wedge dy\wedge (dz_{1}\wedge d\overline{z}_{1}
      + dz_{2}\wedge d\overline{z}_{2} + dz_{3}\wedge d\overline{z}_{3})
      \eqbreak
      + dx\wedge \re(dz_{1}\wedge dz_{2}\wedge dz_{3})
      \eqbreak
      - dy \wedge \im(dz_{1}\wedge dz_{2}\wedge dz_{3})
      - \tfrac{1}{8}(dz_{1}\wedge d\overline{z}_{1}
      + dz_{2}\wedge d\overline{z}_{2} + dz_{3}\wedge d\overline{z}_{3})^{2}
  \end{split}
\end{equation*}
with Killing vector fields
\begin{gather*}
  U_{0} = \D{}{x},\quad U_{1} = \D{}{y},\quad U_{2} =
  2\re\Bigl(i\Bigl(z_{1}\D{}{z_{1}} -
  z_{3}\D{}{z_{3}}\Bigr)\Bigr), \\
  U_{3} = 2\re\Bigl(i\Bigl(z_{2}\D{}{z_{2}} -
  z_{3}\D{}{z_{3}}\Bigr)\Bigr)
\end{gather*}
generating the \( T^{4} \) action.
The components of the corresponding multi-moment map are:
\begin{gather*}
  \nu_{0} = \im(z_{1}z_{2}z_{3}),\quad
  \nu_{1} = \re(z_{1}z_{2}z_{3}),\\
  \nu_{2} = -\tfrac{1}{2}(\abs{z_{2}}^{2}-\abs{z_{3}}^{2}),\quad
  \nu_{3} = \tfrac{1}{2}(\abs{z_{1}}^{2}-\abs{z_{3}}^{2}).
\end{gather*}

For one-dimensional stabiliser the flat model is
\( M = (T^{3} \times \bR) \times \bC^{2} \), with local coordinates
\( x_{1},x_{2},x_{3},u \) for \( T^{3} \times \bR \) and \( (z,w) \)
for \( \bC^{2} \),
\begin{equation*}
  \begin{split}
    \Phi
    &= dx_{1}\wedge dx_{2}\wedge dx_{3} \wedge du
      \eqbreak
      + (dx_{2}\wedge dx_{3} - dx_{1}\wedge du)\wedge
      \tfrac{i}{2}(dz\wedge d\overline{z} + dw\wedge d\overline{w})
      \eqbreak
      - dx_{1}\wedge \re((dx_{2}-idx_{3})\wedge dz\wedge dw)
      \eqbreak
      + du \wedge \im((dx_{2}-idx_{3})\wedge dz\wedge dw)
      \eqbreak
      + \tfrac{1}{8}(dz\wedge d\overline{z} + dw\wedge
      d\overline{w})^{2},
  \end{split}
\end{equation*}
with vector fields
\begin{gather*}
  U_{0} = \D{}{x_{1}},\quad U_{1} = \D{}{x_{2}},\quad U_{2}=
  \D{}{x_{3}}, \\
  U_{3} = -2 \re\Bigl(i\Bigl(z\D{}{z} - w\D{}{w}\Bigr)\Bigr).
\end{gather*}
In this case, the multi-moment map \( \nu=(\nu_0,\nu_1,\nu_2,\nu_3) \)
has
\begin{equation*}
  \nu_{0} = \tfrac{1}{2}(\abs{z}^{2}-\abs{w}^{2}),\quad
  \nu_{1} = -\re(zw),\quad
  \nu_{2} = - \im(zw),\quad
  \nu_{3} = -u.
\end{equation*}

Now let us consider a general \( \Spin(7) \)-manifold \( M \) with
multi\hyphen Hamiltonian \( T^{4} \)-action.
Suppose \( p \in M \) is a point with stabiliser \( T^{2} \).
As \( \Spin(7) \) acts transitively on the space of two-dimensional
subspaces in \( \bR^{8} \), we may identify \( T_{p}M \) with
\( \bR^{2} \times \bC^{3} = T_{0}(T^{2} \times \bC^{3}) \) in such a
way that the \( \Spin(7) \)-forms agree at this point.

The exponential map of \( M \) at \( p \) identifies a neighbourhood
of \( 0 \in T_{0}(T^{2} \times \bC^{3}) \) equivariantly with a
neighbourhood of \( p\in M \).
We may then choose our identifications so that
\( (U_{2})_{p} = 0 = (U_{3})_{p} \),
\( (\nabla U_{2})_{p} = \diag(i,0,-i) \) and
\( (\nabla U_{3})_{p} = \diag(0,i,-i) \), both in \( \su(3) \).
Now note that \( \spin(7) \) contains \( \su(4) \) as a subalgebra and
that \( \spin(7) \cong \su(4) + W \), where
\( W \otimes \bC \cong \Lambda^{2}\bC^{4} \).
It follows that the diagonal maximal torus \( \lie{t}^{2} \) in
\( \su(3) \subset \su(4) \) has all weights on \( W\otimes \bC \)
non-zero and that the centraliser of \( \lie{t}^{2} \) in
\( \spin(7) \) is just the diagonal maximal torus of \( \su(4) \).
For any \( U \) generating \( T^{4}/\Stab(p) \cong T^{2} \), we have
\( (\nabla U)_{p} \) is an element of \( \spin(7) \) commuting with
both \( (\nabla U_{2})_{p} \) and \( (\nabla U_{3})_{p} \),
cf.~\cite[\S4.2.1]{Madsen-S:toric-G2}.
Thus there exist \( a,b \in \bR \) such that
\( (\nabla (U+aU_{2}+bU_{3}))_{p} \) is proportional to
\( \diag(-3i,i,i,i) \).
It follows that we can choose \( U_{0} \) and \( U_{1} \) so that
at~\( p \) they are orthonormal,
\( (\nabla U_{0})_{p} = c\diag(-3i,i,i,i) \), \( c\in\bR \), and
\( (\nabla U_{1})_{p} = 0 \).  But we have
\begin{equation*}
  0 = [U_{0},U_{1}]_{p} = (\nabla_{U_{0}}U_{1})_{p} -
  (\nabla_{U_{1}}U_{0})_{p}
  = 0 + 3c i U_{1} = - 3c U_{0},
\end{equation*}
which implies that \( c = 0 \).
Hence, \( (\nabla U_{0})_{p} = 0 \) too.

Computing covariant derivatives as in~\cite{Madsen-S:toric-G2}, we now
have that \( \nabla^{m}U_{j} \) agrees with the flat model at~\( p \)
for
\begin{equation*}
  (m,j) \in (\Set{0,1} \times \Set{0,1})
  \cup (\Set{0,1,2} \times \Set{2,3}),
\end{equation*}
and is zero for
\( (m,j) \in (\Set{1}\times \Set{0,1}) \cup (\Set{0,2} \times
\Set{2,3}) \).
Thus \( \nabla^{m}\nu_{i} \), which is a sum of terms
\(
\Phi(\nabla^{m_{1}}U_{j},\nabla^{m_{2}}U_{k},\nabla^{m_{3}}U_{\ell},\any)
\) with \( m = 1 + m_{1} + m_{2} + m_{3} \), agrees with the flat
model at \( p \) for
\begin{equation*}
  (m,i) \in (\Set{0,1,2,3,4} \times \Set{0,1})
  \cup (\Set{0,1,2,3} \times \Set{2,3})
\end{equation*}
This exactly matches the degree of agreement we have in the
\( \G_{2} \)-case, and we can apply the analysis
of~\cite[\S4.4]{Madsen-S:toric-G2} to conclude that the multi-moment
map induces a local homeomorphism of the quotient.

\smallbreak Let us now turn to the case when \( p\in M \) has a
stabiliser of dimension one.
Let \( U_{3} \) be a generator for the stabiliser~\( S^{1} \).
Let \( U_{0}, U_{1}, U_{2} \) be any three vector fields generated by
the \( T^{4} \)-action, with the property that they are orthonormal
at~\( p \).
Then the triple-cross product
\( (U_{0} \times U_{1} \times U_{2})_{p} \) is an invariant unit
vector in \( T_{p}M \) that is orthogonal to the orbit.
As \( \Spin(7) \) acts transitively on three-dimensional subspaces of
\( \bR^{8} \) we may identify \( T_{p}M \) with
\( (\bR^{3} \times \bR) \times \bC^{2} = T_{0}(T^{3} \times \bR)
\times \bC^{2} \) in such a way that \( (\nabla U_{3})_{p} \) acts as
\( \diag(i,-i) \in \su(2) \) and the \( \Spin(7) \)-forms agree
at~\( p \).
We have \( d\nu_{3} = - (U_{0} \times U_{1} \times U_{2})^\flat \) is
non-zero at \( p \) and so provides an invariant transverse coordinate
to a seven-dimensional level set through~\( p \).
We have \( (d\nu_{i})_{p} = 0 \), for \( i=0,1,2 \), and
\( \nabla^{2}\nu_{i} \) is determined at \( p \) by
\( U_{0},U_{1},U_{2} \) and \( \nabla U_{3} \) via \( \Phi \), so
these agree with the flat model at this point.
This means that we can once again apply the \( \G_{2} \)-analysis to
conclude that the multi-moment maps provide a local homeomorphism to
\( \bR^{4} \) around \( p \).

Summarising the discussion of this section, we have the following
description of the orbit space of toric \( \Spin(7) \)-manifolds:

\begin{theorem}
  Let \( (M,\Phi) \) be a toric \( \Spin(7) \)-manifold.
  Then \( M/T^4 \) is homeomorphic to a smooth four-manifold.
  Moreover, the multi-moment map \( \nu \) induces a local
  homeomorphism \( M/T^4\to \bR^4 \).  \qed
\end{theorem}

We suspect that the image of the set of special orbits plays an
important role, so it is worthwhile addressing this topic more
explicitly.
First, if \( p\in M \) is a point with stabiliser \( S^1 \), then the
above analysis gives us an integral basis \( U_0,U_1,U_2,U_3 \) of
\( \lt^4 \) such that \( (U_3)_p=0\).
Inspection shows that this holds for all points of \( T^3\times\bR \)
in the flat model.
Hence, the first three components of \( \nu \) are constant on this
set, and the image under \( \nu \) of this family of singular orbits
is a straight line parameterised by \( \nu_3 \).

Next, let us consider the case when \( p \) has \( T^2 \) as its
isotropy subgroup.
Then the normal bundle of the two-torus \( T^4 p \) is \( \bC^3 \),
and there are three families of points with circle stabiliser, meeting
at \( p \).
Again looking at the associated flat model, we see that there is an
integral basis \( U_0,U_1,U_2,U_3 \) of \( \lt^4 \) that has
\( (U_2)_p=0=(U_3)_p \) at \( p \) and such that \( U_2 \), \( U_3 \)
and \( -U_2-U_3 \) generate the circle stabilisers of the three
families.
The images of these families under \( \nu \) have constant \( \nu_0 \)
and \(\nu_1 \) coordinates and give three half-lines meeting at
\( \nu(p) \) and lying in \( \nu_2 \), \( \nu_3\) or \( \nu_2-\nu_3 \)
constant.

Of course, we do not generally know how these lines are aligned in the
target space \( \bR^4 \).

\begin{proposition}
  The image in \( M/T^4 \) of the union \( M\setminus M_0 \) of
  singular orbits consists locally of trivalent graphs in \( \bR^4 \)
  with edges that are straight lines of rational slope in the
  \( \nu \)-coordinates, with primitive slope vectors summing to zero
  at each vertex.  \qed
\end{proposition}

\section{Orthogonal Killing vectors}
\label{sec:diag-case}

In contrast with the \( \G_2 \)-case (see, for example,
\cite[\S5.1.2]{Madsen-S:toric-G2}), there are no known examples of
complete toric \( \Spin(7) \)-manifolds with full holonomy.
On the other hand, one would expect that also in this setting, the
analysis of `diagonal' solutions might lead to simple (incomplete)
explicit metrics with full holonomy.

So let us assume \( V_{ij}=0 \) for all \( i\neq j \), i.e., the
generating vector fields for the torus action are orthogonal.
Writing \( V_i \) for \( V_{ii} \), the \( \Spin(7) \)-metric now
takes the form
\begin{equation*}
  g=\sum_{i=0}^3\frac1{V_i}\bigl(\theta_i^2+V_0V_1V_2V_3d\nu_i^2\bigr).
\end{equation*}

In this case, the curvature \( 2 \)-forms associated with the
\( T^4 \) fibration are given by:
\begin{align*}
  \omega_0 &= - V_{3}\D{V_{0}}{\nu_3}d\nu_{12}
             + V_{2}\D{V_{0}}{\nu_2}d\nu_{13}
             - V_{1}\D{V_{0}}{\nu_1}d\nu_{23},\\
  \omega_1 &= V_{3}\D{V_{1}}{\nu_3}d\nu_{02}
             - V_{2}\D{V_{1}}{\nu_2}d\nu_{03}
             + V_{0}\D{V_{1}}{\nu_0}d\nu_{23},\\
  \omega_2 &= - V_{3}\D{V_{2}}{\nu_3}d\nu_{01}
             + V_{1}\D{V_{2}}{\nu_1}d\nu_{03}
             - V_{0}\D{V_{2}}{\nu_0}d\nu_{13},\\
  \omega_3 &= V_{2}\D{V_{3}}{\nu_2}d\nu_{01}
             - V_{1}\D{V_{3}}{\nu_1}d\nu_{02}
             + V_{0}\D{V_{3}}{\nu_0}d\nu_{12}.
\end{align*}

The divergence-free condition tells us that
\( \partial V_i/\partial\nu_i=0 \), for \( i=0,1,2,3 \).
Then the condition \( d\omega=0 \) is given by the equations
\begin{equation}
  \label{eq:L-red}
  \sum_{j=0}^3V_{j}\Dsq{V_{i}}{\nu_j} = 0,\qquad i=0,1,2,3,
\end{equation}
together with
\begin{equation}
  \label{eq:Q-red}
  \D{V_{\mathstrut i}}{\nu_{j}}\D{V_{j}}{\nu_{i}}=0,\qquad i,j
  \in \Set{0,1,2,3},\ i\ne j.
\end{equation}

\begin{proposition}
  After permuting indices, analytic solutions to
  equation~\eqref{eq:Q-red} with
  \( \partial V_{i}/\partial \nu_{i} = 0 \) for all \( i \), have one
  of the following forms:
  \begin{enumerate}
  \item\label{item:r32} \( V_{0} = V_{0}(\nu_{1},\nu_{2},\nu_{3}) \),
    \( V_{1} = V_{1}(\nu_{2},\nu_{3}) \),
    \( V_{2} = V_{2}(\nu_{3}) \), \( V_{3} \)~constant;
  \item\label{item:r31} \( V_{0} = V_{0}(\nu_{1},\nu_{2},\nu_{3}) \),
    \( V_{1} = V_{1}(\nu_{2}) \), \( V_{2} = V_{2}(\nu_{3}) \),
    \( V_{3} = V_{3}(\nu_{1})\);
  \item\label{item:r23} \( V_{0} = V_{0}(\nu_{1},\nu_{3}) \),
    \( V_{1} = V_{1}(\nu_{2},\nu_{3}) \),
    \( V_{2} = V_{2}(\nu_{0},\nu_{3}) \), \( V_{3} \)~constant;
  \item\label{item:r22} \( V_{0} = V_{0}(\nu_{1},\nu_{2}) \),
    \( V_{1} = V_{1}(\nu_{2},\nu_{3}) \),
    \( V_{2} = V_{2}(\nu_{3}) \), \( V_{3} = V_{3}(\nu_{0}) \);
  \end{enumerate}
\end{proposition}

\begin{proof}
  Define \( r=r(V_{0},\dots,V_{3}) \) to be the largest number such
  that some \( V_{i} \) has \( r \)~partial derivatives
  \( \partial V_{i}/\partial \nu_{j} \) not identically zero.

  If \( r = 3 \), then we reorder indices so that there is a on open
  dense set~\( U \) on which
  \( (\partial V_{0}/\partial \nu_{i})_{p} \ne 0 \) for \( i=1,2,3 \)
  and all \( p\in U \).
  Equation~\eqref{eq:Q-red} then gives
  \( \partial V_{i}/\partial \nu_{0} = 0 \) on \( U_{0} \) for
  \( i=1,2,3 \).
  Thus \( V_{1} = V_{1}(\nu_{2},\nu_{3}) \),
  \( V_{2} = V_{2}(\nu_{1},\nu_{3}) \) and
  \( V_{3}(\nu_{1},\nu_{2}) \).
  Let \( r' = r(V_{1},V_{2},V_{3}) \).
  If \( r' = 2 \), then we rearrange to get
  \( (\partial V_{1}/\partial \nu_{j})_{p} \ne 0 \) for \( j=2,3 \),
  for all \( p \) in a (smaller) open dense set.
  It follows that \( V_{2} = V_{2}(\nu_{3}) \) and
  \( V_{3} = V_{3}(\nu_{2}) \).
  But by~\eqref{eq:Q-red}, only one of these can have a derivative
  that is not identically zero, so we have case~\ref{item:r32}.
  If \( r' = 1 \), then we may assume \( V_{1} = V_{1}(\nu_{2}) \) is
  not constant zero.
  It follows that \( \partial V_{2}/\partial \nu_{1} \equiv 0 \), so
  \( V_{2} = V_{2}(\nu_{3}) \) and we get case~\ref{item:r31}.

  If \( r=2 \), then we may take \( V_{0} = V_{0}(\nu_{1},\nu_{2}) \),
  non-constant in each variable.
  This implies \( V_{1} = V_{1}(\nu_{2},\nu_{3}) \).
  If \( V_{1} \) is non-constant in both variables, then
  \( V_{2} = V_{2}(\nu_{3}) \) and
  \( V_{3} = V_{3}(\nu_{0},\nu_{2}) \).
  Now either \( V_{2} \) is constant, which may be rearranged to
  case~\ref{item:r23}, or \( V_{3} = V_{3}(\nu_{0}) \), which is
  case~\ref{item:r22}.

  For \( r = 1 \), we may assume \( V_{0} = V_{0}(\nu_{1}) \).
  If this is non-constant, then we may take
  \( V_{1} = V_{1}(\nu_{2}) \).
  When \( V_{1} \) is non-constant, we then have \( V_{2} \) is
  \( V_{2}(\nu_{3}) \) or \( V_{2}(\nu_{0}) \).
  In the former case \( V_{3} = V_{3}(\nu_{0}) \), a subcase
  of~\ref{item:r31}, or \( V_{3} = V_{3}(\nu_{1})\), a subcase
  of~\ref{item:r22}.
  The latter case gives subcases of~\ref{item:r31}.
\end{proof}

Solutions to the full \( \Spin(7) \) equations with \( V_{3} \)
constant are just Riemannian products of a circle with a toric
\( \G_2 \)-manifold.
Thus irreducible solutions have to fall under cases~\ref{item:r31}
or~\ref{item:r22} above.
A simple solution of this type is given by taking \( V_0=\nu_1 \),
\( V_1=\nu_2\), \( V_2=\nu_3 \), \( V_3=\nu_0 \), \( \nu_{i} > 0 \)
for all~\( i \).
This gives the following metric for which a curvature computation
shows that its (restricted) holonomy is equal to \( \Spin(7) \):
\begin{equation*}
  \begin{split}
    g
    &= \frac1{\nu_1}\theta_0^2 + \frac1{\nu_2}\theta_1^2 +
      \frac1{\nu_3}\theta_2^2 + \frac1{\nu_0}\theta_3^2 \eqbreak
      + \nu_2\nu_3\nu_0d\nu_0^2 + \nu_1\nu_3\nu_0d\nu_1^2 +
      \nu_1\nu_2\nu_0d\nu_2^2 + \nu_1\nu_2\nu_3d\nu_3^2,
  \end{split}
\end{equation*}
where
\begin{equation*}
  d\theta_0=-\nu_2d\nu_{23},\quad
  d\theta_1=\nu_3d\nu_{30},\quad
  d\theta_2=-\nu_0d\nu_{01},\quad
  d\theta_3=\nu_1d\nu_{12}.
\end{equation*}

In general, if \( V_{i} \) is a function of a single variable, then
equation~\eqref{eq:L-red} forces \( V_{i} \) to be linear in that
variable.
Thus after an affine change of variables, for irreducible solutions in
case~\ref{item:r31}, the equations~\eqref{eq:L-red} become
\begin{equation}
  \label{eq:r31}
  \nu_{2}\Dsq{V_{0}}{\nu_{1}} + \nu_{3}\Dsq{V_{0}}{\nu_{2}} +
  \nu_{1}\Dsq{V_{0}}{\nu_{3}} = 0.
\end{equation}
A simple solution is then \( V_{0} = \nu_{1}\nu_{2}\nu_{3} \) giving
the metric
\begin{equation*}
  \begin{split}
    g
    &= \frac{1}{\nu_{1}\nu_{2}\nu_{3}} \theta_{0}^{2} + \frac{1}{\nu_{2}}
      \theta_{1}^{2} + \frac{1}{\nu_{3}} \theta_{2}^{2} +
      \frac{1}{\nu_{1}} \theta_{3}^{2}
      \eqbreak
      + \nu_{1}\nu_{2}\nu_{3} d\nu_{0}^{2} +
      \nu_{1}^{2}\nu_{2}\nu_{3}^{2}d\nu_{1}^{2} +
      \nu_{1}^{2}\nu_{2}^{2}\nu_{3}d\nu_{2}^{2} +
      \nu_{1}\nu_{2}^{2}\nu_{3}^{2}d\nu_{3}^{2},
  \end{split}
\end{equation*}
where
\begin{gather*}
  d\theta_{0} = - \nu_{1}^{2}\nu_{2} d\nu_{12} - \nu_{3}^{2}\nu_{1}
  d\nu_{31} - \nu_{2}^{2}\nu_{3}d\nu_{23},\\
  d\theta_{1} = - \nu_{3}d\nu_{03},\quad d\theta_{2} = -
  \nu_{1}d\nu_{01},\quad d\theta_{3} = - \nu_{2}d\nu_{02}
\end{gather*}
on \( \nu_{i} > 0 \) for~\( i = 1,2,3 \).
Another solution is obtained by taking
\( V_0 = \nu_1^3 \nu_3 + \nu_2^3 \nu_1 - 2 \nu_3^3 \nu_2 \) on the
non-empty domain where \( V_{0} > 0 \) and \( \nu_{i} > 0 \) for
\( i = 1,2,3 \).

For case~\ref{item:r22}, we have
\begin{equation*}
  V_{1}\Dsq{V_{0}}{\nu_{1}} + \nu_{3}\Dsq{V_{0}}{\nu_{2}} = 0,\quad
  \nu_{3}\Dsq{V_{1}}{\nu_{2}} + \nu_{0}\Dsq{V_{1}}{\nu_{3}} = 0.
\end{equation*}
But this holds for an open set in all the variables.
In the second equation, letting \( \nu_{0} \) vary we see that
\( V_{1} = V_{1}(\nu_{2},\nu_{3}) \) must be linear in each variable,
\( V_{1} = A + B\nu_{2} + C\nu_{3} + D\nu_{2}\nu_{3} \).
Considering the \( \nu_{3} \)-dependence, the first equation decouples
as
\begin{equation}
  \label{eq:r22}
  (A+B\nu_{2})\Dsq{V_{0}}{\nu_{1}} = 0,\quad
  (C+D\nu_{2})\Dsq{V_{0}}{\nu_{1}} + \Dsq{V_{0}}{\nu_{2}} = 0.
\end{equation}
If \( A \) or \( B \) is non-zero, then
\( V_{0} = V_{0}(\nu_{1},\nu_{2}) \) becomes linear in both variables;
otherwise we have \( V_{1} = C\nu_{3} + D\nu_{2}\nu_{3} \) and
\( V_{0} \)~satisfies the last equation of~\eqref{eq:r22}.

\providecommand{\bysame}{\leavevmode\hbox to3em{\hrulefill}\thinspace}
\providecommand{\MR}{\relax\ifhmode\unskip\space\fi MR }
\providecommand{\MRhref}[2]{%
  \href{http://www.ams.org/mathscinet-getitem?mr=#1}{#2}
}
\providecommand{\href}[2]{#2}

\end{document}